\documentclass[10pt]{amsart}
\usepackage{latexsym, amsmath,amssymb, hyperref}

\usepackage{amsthm}

\newcommand{\cal}{\mathcal}

\numberwithin{equation}{section}

\newtheorem{theorem}{Theorem}
\newtheorem{corollary}[theorem]{Corollary}

\newtheorem{lemma}[theorem]{Lemma}
\newtheorem{corr}[theorem]{Corollary}
\newtheorem{conjecture}[theorem]{Conjecture}

\newtheorem{proposition}[theorem]{Proposition}
\newtheorem{deff}[theorem]{Definition}

\newcommand{\bth}{\begin{theorem}}
	\newcommand{\ble}{\begin{lemma}}
		\newcommand{\bcor}{\begin{corr}}

			\newcommand{\bdeff}{\begin{deff}}

				\newcommand{\bprop}{\begin{proposition}}
					\newcommand{\ele}{\end{lemma}}
				\newcommand{\ecor}{\end{corr}}
			\newcommand{\edeff}{\end{deff}}
		
		\newcommand{\eprop}{\end{proposition}}

	\newcommand{\la}{\lambda}
	
	\newcommand{\eps}{\varepsilon}
	\newcommand{\e}{\varepsilon}

	\renewcommand{\Pi}{\varPi}

	\newcommand{\Rt}{{\Bbb R}^2}

	\newcommand{\tidle}{\tilde}

	\newcommand{\R}{{\mathbb R}}
	
	\newcommand{\tg}{\tilde g}
	\newcommand{\dgt}{d_{\tilde g}}
	\newcommand{\sqrtg}{\sqrt{-\Delta_g}}
	\newcommand{\sqrtd}{\sqrt{-\Delta_{\tilde g}}}
	\newcommand{\Pe}{\sqrt{-\Delta}}

	\thanks{ }
	
	\usepackage{graphicx}

\begin{document}

		\title[Improved Generalized Periods Estimates]
		{Improved Generalized Periods estimates on  Riemannian 
			Surfaces with Nonpositive Curvature}

		%
		%
		%
		%
		%
		\begin{abstract}  
			We show that on compact Riemann surfaces of negative curvature, the generalized periods, i.e. the $\nu$-th order Fourier coefficient of eigenfunctions $e_\la$ over a period geodesic $\gamma$ goes to 0 at the rate of $O((\log\la)^{-1/2})$, if $0<\nu<c_0\la$, given any $0<c_0<1$.  No such result is possible for the sphere $S^2$ or the flat torus $\mathbb T^2$.   Combined with the quantum ergodic restriction result of Toth and Zelditch \cite{TZ}, our results imply that for a generic closed geodesic $\gamma$ on a compact hyperbolic surface, the restriction $e_{\la_j}|_\gamma$ of an orthonormal basis $\{e_{\la_j}\}$ has a full density subsequence that goes to zero weakly in $L^2(\gamma)$. Our proof consists of a further refinement of a recent paper by Sogge, Xi and Zhang \cite{Gauss} on the geodesic period integrals ($\nu=0$), which featured the Gauss-Bonnet Theorem as a key quantitative tool to  avoid geodesic rectangles on the universal cover of $M$. In contrast, we shall employ the Gauss-Bonnet Theorem to quantitatively avoid geodesic parallelograms. As in \cite{Gauss}, the use of Gauss-Bonnet also enables us to weaken our curvature condition, by allowing the curvature  to vanish at an averaged	rate of finite type. 
		\end{abstract}
		\author{Yakun Xi}
		\address{Department of Mathematics, University of Rochester, Rochester, NY 14627}
		\email{yxi4@math.rochester.edu}
		

		\maketitle
		
		\section{Introduction}
		Let $e_\la$ denote the $L^2$-normalized eigenfunction on a compact, boundary-less Riemannian surface $(M,g)$, i.e., 
		$$-\Delta_g e_\la=\la^2 e_\la, \quad \text{and } \, \, \int_M |e_\la|^2 \, dV_g=1.$$
		Here $\Delta_g$ denotes the Laplace-Beltrami operator on $(M,g)$ and $dV_g$  is the volume element associated with metric $g$.
		
		It is an area of interest in both number theory and harmonic analysis to study various quantitative behaviors of the eigenfunctions restricted to certain smooth curve, among which the geodesic period integral of eigenfunctions on compact hyperbolic surfaces have been studied intensively due to its significance in number theory.

		Using the Kuznecov formula, Good \cite{Good} and Hejhal \cite{Hej} independently showed that if $\gamma_{per}$ is a periodic geodesic on a compact hyperbolic
		surface $M$ parametrized by arc-length, then, uniformly in $\la$,
		\begin{equation}\label{i.1}
		\Bigl|\, \int_{\gamma_{per}} e_\la \, ds\, \Bigr| \le C_{\gamma_{per}}.
		\end{equation}
		This result was later generalized by Zelditch~\cite{ZelK}, who showed that if $\la_j$ are the eigenvalues of $\sqrt{-\Delta_g}$ for an orthonormal basis of eigenfunctions $e_{\la_j}$ on
		a compact Riemannian surface, and if $p_j(\gamma_{per})$ denote the period integrals  of $e_{\la_j}$ as in \eqref{i.1},
		then
		\begin{equation}\label{Zel}\sum_{\la_j\le \la}|p_j(\gamma_{per})|^2 =c_{\gamma_{per}} \la +O(1),\end{equation}
		which is stronger than \eqref{i.1}.  Further work for hyperbolic surfaces giving more information about the lower order remainder in terms
		of geometric data of $\gamma_{per}$ was done by Pitt~\cite{Pitt}.  Since, by Weyl's Law, the number of eigenvalues (counting multiplicities) that are smaller than
		$\la$ is about $\la^2$, \eqref{Zel} implies that, at least on average, one can do much better than \eqref{i.1}.  The problem
		of improving this upper bound for hyperbolic surfaces was raised and discussed in Pitt~\cite{Pitt} and Reznikov~\cite{Rez}. Indeed, it was observed in \cite{CSPer} that \eqref{i.1} is sharp on compact
		Riemannian surfaces of constant non-negative curvature.  On the round sphere $S^2$, the integrals in \eqref{i.1} have unit size if
		$\gamma_{per}$ is the equator and $e_\la$ is an $L^2$-normalized zonal spherical harmonic of even degree.  Also on the flat torus ${\mathbb T}^2$, for every
		periodic geodesic, $\gamma_{per}$, one can find a sequence of eigenvalues $\la_k$ and eigenfunctions $e_{\la_k}$ so that
		$e_{\la_k}\equiv 1$ on $\gamma_{per}$ and $\|e_{\la_k}\|_{L^2({\mathbb T}^2)}\approx 1$.  In contrast, as an analogy with the Lindel\"of conjecture for certain $L$-functions, it is conjectured by Reznikov that the geodesic period integrals on a compact hyperbolic surface satisfy the following:
		\begin{conjecture}[\cite{Rez}]Let $\gamma_{per}$ be a periodic geodesic on a compact hyperbolic surface $(M,g)$. Then given $\eps>0$, there exists a constant $C_\eps$ depending on $\eps$, $M$ and the length of $\gamma_{per}$, such that 
				\begin{equation}	\Bigl|\, \int_{\gamma_{per}} e_\la \, ds\, \Bigr| \le C_\eps\la^{-\frac12+\eps}.
\end{equation}		
	
	\end{conjecture}

		In a paper of Chen and Sogge \cite{CSPer}, the first improvement over \eqref{i.1} for surfaces with negative curvature was obtained. It was shown that the period integrals in \eqref{i.1} converge to 0 as $\la\to \infty$ if $(M,g)$ has strictly negative
		curvature.  The proof exploited the simple geometric fact that, due to the presence of negative curvature, there is no non-trivial geodesic rectangle on the universal cover of $M$.  This allowed them to show that the period integrals \eqref{i.1} goes to 0 as $\la\rightarrow\infty$, by using a stationary
		phase argument involving a reproducing kernel for the eigenfunctions. In a recent joint paper of Sogge, the author and Zhang \cite{Gauss}, this method was further refined, and they managed to show that 
		\begin{equation}\label{gauss}
		\Bigl|\, \int_{\gamma_{per}} e_\la \, ds\, \Bigr| = O((\log\lambda)^{-\frac12}),
		\end{equation}
		under the assumption that the curvature $K=K_g$ of $(M,g)$ is  non-positive but allowed to vanish at an averaged
		rate of finite type, in the following sense.
		\begin{deff}[Curvature Condition] Let $(M,g)$ be a surface with non-positive curvature.
		We say that the curvature $K$ of $(M,g)$ satisfies the averaged vanishing condition, if there exists a constant $k>0$, such that whenever $B_r\subset M$ is a geodesic ball of radius $r\le1$ with arbitrary center,  we have that
		\begin{equation}\label{i.2}
		\int_{B_r}K\, dV_g \le -\delta r^k, \quad\text{ for all }r\le 1.
		\end{equation} 
	\end{deff}
		 If the curvature of $M$ is strictly negative, then by the compactness of $M$, one can see that $K\le -\delta$ everywhere for some constant $\delta>0$, and thus clearly \eqref{i.2} is valid. Furthermore,  the averaged vanishing condition holds, if the curvature is negative off of a lower dimensional set of points where it vanishes to finite order.
		 
		The key idea of \cite{Gauss} was to use the Gauss-Bonnet Theorem to get a quantitative version of the ideas used in \cite{CSPer}, that is, to quantitatively avoid geodesic rectangles on the universal cover.  Using Jacobi fields, Wyman (\cite{emmett2}, \cite{emmett1}) generalized the results in \cite{CSPer} and \cite{Gauss} to curves satisfying certain interesting curvature assumptions. See also the recent work of Canzani, Galkowski \cite{canzani} for $o(1)$ bounds for integrals over submanifolds under weaker general geometric assumptions.
		
		Another way of looking at the period integral \eqref{i.1}  is to regard it as the 0-th order Fourier coefficient of $e_\la|_{\gamma_{per}}$. The general $\nu$-th order Fourier coefficients of $e_\la|_{\gamma_{per}}$ are called {\it generalized periods} (see \cite{Rez}). Generalized periods for hyperbolic surfaces naturally arise in the theory
		of automorphic functions, and are of interest in their own right, thus have been studied considerably by number theorists.  It was shown by Reznikov \cite{Rez} that on compact hyperbolic surfaces, if $\gamma$ is a periodic geodesic or a geodesic circle, the $\nu$-th order Fourier coefficients of $e_\lambda|_\gamma$ is uniformly bounded if $\nu\le c_\gamma\la$ for some constant $c_\gamma$ depending on $\gamma$.  In this spirit, it was later proved by the author \cite{inner} that Reznikov's bounds actually hold for arbitrary smooth closed curves over arbitrary compact Riemannian surfaces.
		\begin{theorem}[\cite{inner}]
			Let $\gamma$ be a smooth closed curve on $(M,g)$ parametrized by arc-length. Let $|\gamma|$ denote its length. Given $0<c_0<1$, if $\nu$ is an integer multiple of $2\pi|\gamma|^{-1}$ such that $0\le\frac{\nu}{\la}<c_0<1$,  then we have
			\begin{equation}\label{period}
			\Big|\int_{\gamma}e_\lambda(\gamma(s)) e^{-i\nu s}\, ds\Big|\le C |\gamma|,
			\end{equation}
			where the constant $C$ only depends on $(M,g)$ and $c_0$.
			In addition, if $\mu>c_0^{-1}\la$, then we have
			\begin{equation}\label{period'}
			\Big|\int_{\gamma}e_\lambda(\gamma(s)) e^{-i\nu s}\, ds\Big|\le C_N \nu^{-N},
			\end{equation}
			for all $N\in\mathbb N.$
		\end{theorem}
Like the period integrals, the above bounds are sharp on both the sphere $S^2$ and the flat torus $\mathbb T^2$. (See \cite[Section 5]{inner}.)	We also remark that our frequency gap condition $0\le\frac{\nu}{\la}<c_0<1$ is necessary, in the sense that \eqref{period} is in general not possible at the resonant frequency $\nu=\la$. Indeed, on $S^2$,  $L^2$-normalized highest weight spherical harmonics with frequency $\la$ restricted to the equator have $\nu$-Fourier coefficient $\sim\nu^\frac14$, which represents a big jump from both \eqref{period} and \eqref{period'}.
		Another key insight provided by 
		\cite{Gauss} is that under the curvature assumption \eqref{i.2}, the $\nu$-th Fourier coefficients of $e_\lambda|_{
		\gamma_{per}}$ satisfies
		\begin{equation}\label{gauss'}
		\Bigl|\, \int_{\gamma_{per}} e_\la e^{-i\nu s} \, ds\, \Bigr| \le C_\nu|\gamma_{per}|(\log\lambda)^{-\frac12},
		\end{equation} where $C_\nu$ is a constant times some positive power of $\nu$. On the other hand, it is also conjectured in \cite{Rez} that for compact hyperbolic surface, we should expect much better estimates.
				\begin{conjecture}[\cite{Rez}]\label{C2}Let $\gamma_{per}$ be a fixed closed geodesic on a compact hyperbolic surface $(M,g)$. Then given $\eps>0$, $0<c<1$, there exists a constant $C_\eps$ depending on $\eps$, $M$ and the length of $\gamma_{per}$, such that for $0\le\frac{\nu}{\la}<c_0<1$, we have
			\begin{equation}	\Bigl|\, \int_{\gamma_{per}} e_\la(\gamma_{per}(s)) e^{-i\nu s} \, ds\, \Bigr| \le C_\eps\la^{-\frac12+\eps}.
			\end{equation}		
			
		\end{conjecture}
		The above conjecture, if true, illustrates the huge differences between the sphere/ torus case and the hyperbolic  case. The curvature of the surface being negative somehow ``filters" out almost all lower frequency oscillations of eigenfunctions over a closed geodesic. Even though Conjecture \ref{C2} seems far beyond
		 with current techniques,  we shall take a first step toward it. To be more specific, we shall prove the analog of \eqref{gauss} for generalized periods, by showing that one may take the constant $C_\nu$ in \eqref{gauss'} to be uniform in $\nu$, assuming $0\le\frac{\nu}{\la}<c_0<1$.
		As in \cite{Gauss}, we shall assume through out our paper that the curvature of our surface $(M,g)$ satisfies the averaged vanishing condition \eqref{i.2}. Our main result is the following.
		\begin{theorem}[Main Theorem]\label{main}
			Let $\gamma=\gamma_{per}$ be a periodic geodesic on a Riemannian surface $(M,g)$ with curvature $K$ satisfying the averaged vanishing conditions \eqref{i.2}. Given $0<c_0<1$,  if $\nu$ is an integer multiple of $2\pi|\gamma|^{-1}$ such that $0\le\frac{\nu}{\la}=\epsilon<c_0<1$,  then we have
			\begin{equation}\label{periods}
			\Big|\int_{\gamma}e_\lambda(\gamma(s)) e^{-i\nu s}\, ds\Big|\le C |\gamma|(\log\la)^{-\frac12},
			\end{equation}
			where the constant $C$ only depends on $M$ and $c_0$.
		\end{theorem}
		It is clear that Theorem \ref{main} is an improvement over both \eqref{period} and \eqref{gauss}, where it is better than \eqref{period} for negatively curved surfaces, and contains \eqref{gauss'} as the special case  when $\nu=0$. To the author's best knowledge, Theorem \ref{main} is the first result showing generalized periods converge to 0 uniformly for $\nu<c_0\la$ for compact hyperbolic surfaces.  Also we remark that, like in \cite{Gauss}, our result \ref{periods} also holds for quasi-modes $\psi_\la$ satisfying (see \cite{SZQuas})
		\begin{equation}\label{2.5}
		(\log\la/\la)\|(\Delta_g+\la^2)\psi_\la\|_{L^2(M)}+\|\psi_\la\|_{L^2(M)}\le 1.
		\end{equation}

Our result also says something about the weak $L^2(\gamma)$ limit of restricted eigenfunctions. The result of \cite{CSPer} indicates that, on negatively curved surfaces, eigenfunctions restricted to a closed geodesic converge to zero in the sense of distributions. In a talk given by Sogge, he proposed the question that whether this is also true in the weak-$L^2(\gamma)$ sense. Our corollary gives a first result regarding this question.

\begin{corollary} \label{co}	Let $\gamma=\gamma_{per}$ be a periodic geodesic on a Riemannian surface $(M,g)$ with curvature $K$ satisfying the averaged vanishing conditions \eqref{i.2}. Let $\{e_{\la_j}\}$ be a sequence of eigenfunctions with bounded $L^2(\gamma)$ norm, then $e_{\la_j}|_\gamma\rightarrow 0$ weakly in  $L^2(\gamma)$.
\end{corollary} 

By the principle of uniform boundedness, a weakly convergent sequence in a Hilbert space must be bounded, thus the above corollary actually implies that $e_{\la_j}|_\gamma\rightarrow 0$ weakly in $L^2(\gamma)$ if and only if they have bounded $L^2(\gamma)$ norms.
Combined with the quantum ergodic restriction (QER) theorem of Toth and Zelditch \cite{TZ}, which says that we have quantum ergodic restriction for generic closed geodesics on a compact hyperbolic surfaces, we have the following.

\begin{corollary} Let $(M,g)$ be a compact hyperbolic surface, $\{e_{\la_j}\}$ an orthonormal basis of eigenfunctions of frequency $\la_j$. Then for a generic\footnote{In the sense of \cite{TZ}} closed geodesic $\gamma$, there exists a density one subsequence $ \{e_{\la_{j_k}}\}$ such that $ e_{\la_{j_k}}|_\gamma\rightarrow 0$ weakly in $L^2(\gamma)$.
\end{corollary} 

The proof of Theorem \ref{main} will follow the general scheme introduced in \cite{CSPer} and \cite{Gauss} rather closely.  In \cite{Gauss}, Gauss-Bonnet Theorem was used to exploit the defects of geodesic quadrilaterals that arise in these
arguments. This allows one to quantitatively avoid geodesic rectangles, which in turn provides favorable control of lower bounds for first and second derivatives of the phase functions 
occurring in the stationary phase arguments. In the current paper, we shall further refine this strategy, by essentially using the Gauss-Bonnet Theorem to quantitatively avoid geodesic parallelograms. We remark that, the strategy in \cite{emmett2} should also provide some nontrivial results for improved generalized periods estimates over curves that satisfy certain  curvature conditions. (e.g. geodesic circles on $\mathbb T^2$ certainly work.) However, it seems that an analog of the curvature condition in \cite{emmett2} has to involve the frequency ratio $\epsilon=\nu/\la$.

Our proof represents the intermediate case between the $\log$-improved period integral estimates in \cite{Gauss} and the $\log$-improved $L^2$ restriction bounds of Blair and Sogge \cite{BSTop}, which says that
\begin{equation}\label{top1}
\|e_\lambda\|_{L^2(\gamma)}\le C\frac{\lambda^\frac14}{(\log\lambda)^\frac12}\|e_\lambda\|_{L^2(M)}.
\end{equation}
Their proof relied on the Topnogov Comparison Theorem, which provided a favorable control over a microlocalization near the tangential direction. (See \cite{chen}, \cite{ChenS}, \cite{xizhang} and \cite{Blair} for further developments on improved $L^p$ restriction estimates.) In contrast, the main contribution in our estimates will be coming from microlocal directions that are strictly in between the tangential case ($L^2$ restriction) and the orthogonal  case (period integrals).

This paper is organized as follows.  In the next section we shall perform a standard reduction using the reproducing kernels for eigenfunctions (see \cite{SFIO}).   By lifting the calculations to the universal cover $(\mathbb R^2,\widetilde g)$, we reduce the proof of Theorem \ref{main} to estimating
oscillatory integrals over geodesics in $(\mathbb R^2,\widetilde g)$. In \S 3,  we employ Gauss-Bonnet Theorem to prove some crucial geometric facts, and then use them to obtain favorable bounds for various derivatives of the phase function of the aforementioned oscillatory integrals. 
In \S 4 we conclude the proof of Theorem \ref{main} using a modified stationary phase lemma developed in \cite{Gauss}, and give the proof of Corollary \ref{co}

In what follows, by possibly rescaling the metric, we shall assume that the injectivity radius of $M$, $\text{Inj}\,M$, is at least 10. We shall always use the letter $\epsilon$ to denote the frequency ratio $\nu/\la$, and $c_0$ a positive constant that is strictly less than 1. The letter $C$ will be used to denote various positive constants depending on $(M,g)$ and $c_0$, whose value could change from line to line.

\textit{Acknowledgments.}
The author would like to thank Professor Sogge, Professor Greenleaf and Professor Iosevich for their constant support and mentoring. It is a pleasure for the author to thank his collaborators Sogge and Zhang, since this paper would not have been possible without their joint work \cite{Gauss}. The author also would like to thank Professor Zelditch for suggesting this problem.

		\section{A Standard Reduction: Lifting to the Universal Cover}
		Let us choose a function $\rho\in {\mathcal S}(\R)$ satisfying
		$$\rho(0)=1 \quad \text{and } \, \, \Hat \rho(\tau)=0 , \quad |\tau|\ge 1/4.$$
		Then for any $T>0$, the Fourier multiplier operator $\rho(T(\la-\sqrt{-\Delta_g}))$ defined by the spectral theorem reproduces eigenfunctions, in the sense that $\rho(T(\la-\sqrt{-\Delta_g}))e_\la=e_\la$. Therefore, in order to prove \eqref{periods}, it suffices to show that we can choose $T=c\log\la$ for some small constant $c=c_M$, so that for $0\le\frac{\nu}{\la}=\epsilon<c_0<1$, we have the uniform bounds
		\begin{equation}\label{2.1}
		\Bigl|\, \int b(t) e^{-i\nu t} \bigl(\rho(T(\la-\sqrt{-\Delta_g}))f\bigr)(\gamma(t))\, dt \, \Bigr|\le C_{M,b}\, (\log\la)^{-1/2}\, \|f\|_{L^2(M)},
		\end{equation}
		where $b(t)$ is a fixed smooth bump function in $C_0^\infty((-1/2,1/2))$.

		To set up the proof of \eqref{2.1} we first note that the kernel of the operator thereof is given by
		$$\rho\bigl(T(\la-\sqrt{-\Delta_g})\bigr)(x,y)=\sum_j \rho(T(\la-\la_j)) e_j(x)\overline{e_j(y)}.$$
		Hence, by Cauchy-Schwarz inequality, we would have \eqref{2.1} if we could show that
		$$\int_M\Bigl|\, \int b(t)\,e^{-i\nu t} \sum_j \rho\bigl(T(\la-\la_j)\bigr) \, e_j(\gamma(t)) \overline{e_j(y)} \, dt \, \Bigr|^2 \, dV_g(y)
		\le C_{M,b}(\log\la)^{-1},$$
		By orthogonality, if $\chi(\tau)=(\rho(\tau))^2$, this is equivalent to showing that if
		$$b(t,s)=b(t)\overline{b(s)}\in C^\infty_0((-1/2,1/2)^2),$$
		then
		\begin{equation}\label{2.7}
		\Bigl|\, \iint b(t,s) e^{i\nu(s-t)}\sum_j \chi(T(\la-\la_j)) e_j(\gamma(t)) \overline{e_j(\gamma(s))} \, dt ds\, \Bigr|
		\le C_{M,b}(\log \la)^{-1}.
		\end{equation}

		Note that by Fourier inversion formula, we have
		$$\sum_j \chi(T(\la-\la_j))e_j(x)\overline{e_j(y)} =\frac1{2\pi T} \int \hat \chi(\tau/T) e^{i\tau \la} 
		\bigl(e^{-i\tau \sqrt{-\Delta_g}}\bigr)(x,y) \, d\tau.$$
		As a first step of proving \eqref{2.7}, let us separate the local and global contributions of \eqref{2.7}. We shall do this by fixing  a bump function $\beta\in C^\infty_0(\R)$ satisfying
		$$\beta(\tau)=1, \, \, \, |\tau|\le 3 \quad \text{and } \, \, \beta(\tau)=0, \, \, \, |\tau|\ge 4.$$
		Then the proof of Lemma 5.1.3 in \cite{SFIO} shows that, since we are assuming that $\text{Inj}\,M\ge10$, we can write
		\begin{multline}\label{2.8}
		\frac1{2\pi T} \int \beta(\tau) \hat \chi(\tau/T) e^{i\tau \la} \bigl(e^{-i\tau\sqrt{-\Delta_g}}\bigr)(x,y) \, d\tau
		\\
		=\frac{\la^{1/2}}T \sum_\pm a_\pm(\la; d_g(x,y))e^{\pm i \la d_g(x,y)} +O(T^{-1}),
		\end{multline}
		if $d_g$ denotes the Riemannian distance on $(M,g)$, where
		\begin{equation}\label{2.9}
		\Bigl|\, \frac{d^j}{dr^j} a_\pm(\la;r)\, \Bigr|\le C_j r^{-j-1/2} \quad \text{if } \, \, r\ge \la^{-1},
		\end{equation}
		and
		\begin{equation}\label{2.10}
		|a_\pm(\la;r)|\le C\la^{1/2} \quad \text{if } \, \, \, 0\le r\le \la^{-1}.
		\end{equation}
		
		Since $d_g(\gamma(t),\gamma(s))=|t-s|$, a simple integration by parts argument shows that \eqref{2.9} and \eqref{2.10} imply  
				$$\la^{1/2}\Bigl|\, \iint b(t,s)e^{\pm i(\la|t-s|+\nu(s-t))} a_\pm(\la;|t-s|) \, dt ds \, \Bigr|\le C_{M,b}.$$
		Therefore, by \eqref{2.8} and the fact that $\nu<c_0\la$ with $0<c_0<1$, for any given $c>0$, we have
		\begin{multline}\label{2.11}
		\frac1{2\pi T}\Bigl|\, \iiint b(t,s) e^{i\nu(s-t)} \beta(\tau)\hat \chi(\tau/T) e^{i\tau\la} \bigl(e^{-i\tau\sqrt{-\Delta_g}}\bigr)(\gamma(t),\gamma(s)) \, d\tau dt ds\, \Bigr|
		\\
		\le C_{M,b}(\log\la)^{-1}, \quad \text{if } \, \, T=c\log\la,
		\end{multline}
		
		In view of \eqref{2.11}, we conclude that we would have \eqref{2.7} if we could obtain the following bounds for
		the remaining part of $\chi(T(\la-\sqrt{-\Delta_g}))$:
		\begin{multline}\label{2.12}
		\frac1{2\pi T}\Bigl| \iiint b(t,s)e^{i\nu(s-t)}\bigl(1-\beta(\tau)\bigr) \, \hat \chi(\tau/T) e^{i\tau\la}
		\bigl(e^{-i\tau\sqrt{-\Delta_g}}\bigr)(\gamma(t),\gamma(s)) \, d\tau ds dt \, \Bigr|
		\\
		\le C_{M,b}(\log\la)^{-1},
		\end{multline}
		if $T=c\log\la$, for some $c\ll 1$ to be specified later.
		Note that for $T\ge 1$ we have the uniform bounds
		$$\frac1{2\pi T}\Bigl| \, \int \bigl(1-\beta(\tau)\bigr) \hat \chi(\tau/T) e^{i\tau (\la+\la_j)} \, d\tau\, \Bigr|
		\le C_N(1+|\la+\la_j|)^{-N}, \, \, \, N\in\mathbb N,$$
		and so, since $\la_j\ge 0$ and $\la\gg 1$,
		$$\frac1{2\pi T}\Bigl| \, \int \bigl(1-\beta(\tau)\bigr) \hat \chi(\tau/T) e^{i\tau \la}
		\bigl(e^{i\tau \sqrt{-\Delta_g}}\bigr)(\gamma(t),\gamma(s)) \, d\tau\, \Bigr|
		\le C_N(1+\la)^{-N}, \, \, \, N\in\mathbb N.$$
		Thus, by Euler's formula, to prove \eqref{2.12}, it suffices to show that if $T=c\log\la$, (for suitable choice of $c=c_M>0$) we have
		\begin{equation}\label{2.13}
		\Bigl| \, \iiint b(t,s) e^{i\nu(s-t)} \bigl(1-\beta(\tau)\bigr)\hat \chi(\tau/T) e^{i\tau \la} \bigl(\cos \tau \sqrt{-\Delta_g}\bigr)(\gamma(t),\gamma(s)) \, d\tau dt ds \, \Bigr|
		\le C_{M,b}.
		\end{equation}
		Here $\bigl(\cos \tau\sqrt{-\Delta_g}\bigr)(x,y)$ is the wave kernel for the map $$C^\infty(M)\ni f\to u\in C^\infty(\R\times M)$$ solving the Cauchy problem
		with initial data $(f,0)$, i.e.,
		\begin{equation}\label{Cauchy}\bigl(\partial_\tau^2-\Delta_g\bigr)u=0, \quad u(0, \, \cdot \, )=f, \quad \partial_\tau u(0, \, \cdot \, )=0.\end{equation}
		
		To be able to compute the integral in \eqref{2.13} we need to relate the wave kernel on $M$ to the corresponding waver kernel on the universal cover of $M$. By the Cartan-Hadamard Theorem  (see e.g. \cite[Chapter 7]{doCarmo}), we can lift the calculations up to the universal cover $(\mathbb{R}^2,\tilde g)$ of $(M,g)$.
		
		Let $\Gamma$ denote the group of deck transformations preserving the associated covering map $\kappa :{\mathbb{R}^2} \to M$ coming from the exponential map about the point $\gamma(0)$, which is chosen to be  the midpoint of the 
		geodesic segment $\{\gamma(t): \, |t|\le \tfrac12\}$. The metric $\tilde g$ is its pullback via the covering map $\kappa$. We shall measure the distances in $(\mathbb{R}^2,\tilde g)$ using its Riemannian distance function $d_{\tilde g}({}\cdot{},{}\cdot{})$. We choose a Dirichlet fundamental domain, $D \simeq M$, centered at the lift $\tilde \gamma(0)$ of $\gamma(0)$, which has the property that $\Rt$ is the disjoint union of the
		 $\alpha(D)$ as $\alpha$ ranges over $\Gamma$ and $\{\tilde y\in \Rt: \, \dgt(0,\tilde y)<10\} \subset D$ since we are assuming that $\text{Inj}\,M\ge10$.   It then follows that we can identify every point $x\in M$ with the unique point $\tilde x\in D$ having the property
		 that $\kappa(\tilde x)=x$.  Let also $\{\tilde \gamma(t)\in\mathbb R^2: |t|\le \tfrac12\}$ similarly denote the set of points in $D$ corresponding to our geodesic segment
		 $\gamma$ in $M$.  Then $\{\tilde \gamma(t): |t|\le \tfrac12\}$ is a line segment on $\mathbb R^2$ of unit length whose midpoint is the origin, and we shall denote just by
		 $\tilde \gamma$ the line through the origin containing this segment. Note that $\tilde \gamma$ then is a geodesic in $\Rt$ for the metric $\tg$, and 
		the Riemannian distance between two points on $\tilde \gamma$ agrees with their Euclidean distance.
		Finally, if $\Delta_{\tg}$ denotes the Laplace-Beltrami operator associated to $\tg$, then, since solutions of the  Cauchy problem \eqref{Cauchy} for $(M,g)$ correspond exactly
		to  $\Gamma$-invariant solutions of the corresponding Cauchy problem associated to the lifted wave operator $\partial^2_t-\Delta_{\tg}$, we have the following
		Poisson formula relating the wave kernel on $(M,g)$ to the one for the universal cover $(\Rt,\tg)$:
		\begin{equation}\label{2.14}
		\bigl(\cos \tau \sqrtg \big)(x,y)=\sum_{\alpha\in \Gamma}\bigl(\cos \tau \sqrtd  \bigr)(\tilde x, \alpha(\tilde y)).
		\end{equation}
		
		Due to this formula, we would have \eqref{2.13} if we could show that for $T=c\log\la$,
		\begin{multline}\label{2.15}
		\sum_{\alpha\in \Gamma} \Bigl| \, \iiint b(t,s)\,e^{i\nu(s-t)} \bigl(1-\beta(\tau)\bigr) \hat \chi(\tau/T) e^{i\tau \la}
		\bigl(\cos \tau \sqrtd \bigr)(\tilde \gamma(t),\alpha(\tilde \gamma(s))) \, d\tau dt ds \, \Bigr|
		\\
		\le C_{M,b}.
		\end{multline}
		
		By Huygens' Principle, the wave kernel $\bigl(\cos \tau \sqrtd \bigr)(\tilde x,\tilde y)$ vanishes identically if $d_{\tilde g}(\tilde x,\tilde y)>\tau$, where $d_{\tilde g}$ denotes the
		Riemannian distance on $(\Rt,\tilde g)$.  Since $\chi=\rho^2$, and our assumption that $\rho(\tau)=0$ for $|\tau|\ge 1/4$ implies that the integrand in
		\eqref{2.15} vanishes when $|\tau|\ge T/2$.  Therefore, since there are $O(\exp(C_MT))$ ``translates'' $\alpha(D)$ of $D$ satisfying $d_{\tilde g}(D,\alpha(D))<T$,
		we conclude that the sum in \eqref{2.15} involves $O(\exp(C_MT))$ nonzero terms.  Thus, in order to prove Theorem \ref{main}, it suffices to show that for large enough $\la$,
					\begin{multline}\label{2.16}
		\Bigl| \, \iiint b(t,s) e^{i\nu(s-t)} \bigl(1-\beta(\tau)\bigr) \hat \chi(\tau/T) e^{i\tau \la}
		\bigl(\cos \tau \sqrtd \bigr)(\tilde \gamma(t),\alpha(\tilde \gamma(s))) \, d\tau dt ds \, \Bigr|
		\\
		\le C_{M,b} \la^{-\sigma_M} \quad \text{if } \, \, \, T=c\log \la,
		\end{multline}
		We shall need an explicit expression for the kernel
		\begin{equation}\label{5.1}
		K_{T,\la}(x,y)=\int \bigl(1-\beta(\tau)\bigr)\hat \chi(\tau/T) e^{i\tau \la} \bigl(\cos \tau \Pe\bigr)(x,y) \, d\tau,
		\end{equation}
		evaluating at $(x,y)=(\tilde \gamma(t),\alpha(\tilde \gamma(s)))$.
		\begin{proposition}[{\cite[Proposition 5.1]{Gauss}}]\label{prop5.1}  Let $T=c\log\la$. If $d_{\tilde g}\ge1$  and $\la\gg 1$, we have
			\begin{equation}\label{5.2}
			K_{T,\la}(x,y)= \la^{1/2}\sum_\pm a_\pm(T,\la; x,y) e^{\pm i\la d_{\tilde g}(x,y)} +R_{T,\la}(x,y),
			\end{equation}
			where
			\begin{equation}\label{5.3}
			|a_\pm(T,\la; x,y)|\le C,
			\end{equation}
			and
			if $\ell=1,2,3,\dots$ is fixed 
			\begin{equation}\label{5.4}
			\Delta^\ell_x a_\pm(T,\la; x,y) =O(\exp(C_\ell d_{\tilde g}(x,y)))
			\end{equation}
			or
			\begin{equation}\label{5.5}
			\Delta^\ell_y a_\pm(T,\la; x,y) =O(\exp(C_\ell d_{\tilde g}(x,y))),
			\end{equation}
			and
			\begin{equation}\label{5.6}
			|R_{T,\la}(x,y)|\le\la^{-1},
			\end{equation}
			provided the constant $c>0$  is sufficiently small.
			Also,  in this case we also have
			\begin{equation}\label{5.7}
			K_{T,\la}(x,y)=O(\la^{-1}), \quad \text{if } \, \, \, d_{\tilde g}(x,y)\le 1.
			\end{equation}
		\end{proposition}

		Proposition \ref{prop5.1} goes back to the work of B\'erard \cite{Berard} on improved remainder estimates for the Weyl Law, which was proved using the Hadamard parametrix for the wave operator on the universal cover. The version we are quoting here is in \cite{Gauss}. Clearly \eqref{5.7} implies that \eqref{2.16} is valid when $\alpha=Id$.  For $\alpha\neq Id$,  our assumption that $\text{Inj}\,M\ge10$
		guarantees that $d_{\tilde g}(\tilde \gamma(t),\alpha(\tilde \gamma(s)))\ge 1$ in this case, so we can
		use \eqref{5.2}--\eqref{5.6}.
		If we let
		\begin{equation}\label{5.18'}
		\phi_\pm(\alpha;t,s)=\epsilon(s-t)\pm d_{\tilde g}(\tilde \gamma(t),\alpha(\tilde \gamma(s))),
		\end{equation}
		where $\epsilon=\nu/\la$, and $\epsilon<c_0$. Without loss of generality, from now on,  we shall only consider the contribution coming from the term with phase function:
		\begin{equation}\label{5.18}
		\phi(\alpha;t,s)=\phi_+(\alpha;t,s)=\epsilon(s-t)+d_{\tilde g}(\tilde \gamma(t),\alpha(\tilde \gamma(s))),
		\end{equation}
		and amplitude $$a(T,\la; \tilde \gamma(t),\alpha(\tilde \gamma(s)))=a_+(T,\la; \tilde \gamma(t),\alpha(\tilde \gamma(s))).  $$
		The other term with minus signs can be handled in a similar fashion. Indeed, by possibly reversing the orientation of $\gamma$, one can see that there is no essential difference between these two cases.
		The main term for the integral in \eqref{2.16} is 
			\begin{equation}\label{2.16'}
		\la^{\frac12} \iint b(t,s)   a(T,\la; \tilde \gamma(t),\alpha(\tilde \gamma(s))) e^{ i\la \phi(\alpha;t,s)} \,dt ds,
		\end{equation}
		while, by \eqref{5.6}, the remaining term enjoys much better estimates than that in \eqref{2.16}.
		Now we will have \eqref{2.16}, and thus finishing the proof of our main theorem, if we can prove the following:
		
		\begin{proposition}\label{prop2.1}  Given $(M,g)$ with curvature $K$ satisfying \eqref{i.2}, we can fix $c=c_M>0$ so that for $\alpha\neq Id$, $\la\gg 1$, we have
	\begin{equation}\label{2.21}
\Bigl| \, \la^\frac12\iint b(t,s)   a(T,\la; \tilde \gamma(t),\alpha(\tilde \gamma(s))) e^{ i\la \phi(\alpha;t,s)} \,dt ds \Bigr|
\\
\le C_{M,b} \la^{-\sigma_M}, \quad \text{if } \, \, \, T=c\log \la,
\end{equation}
			for some $\sigma_M>0$ which only depends on $(M,g)$.
		\end{proposition} 
		To prove Proposition \ref{prop2.1}, we shall need some geometric tools which will provide us favorable estimates for the various derivatives of our phase function $\phi(\alpha;t,s)$. The key idea is that the Gauss-Bonnet Theorem serves as a  tool to quantitatively eliminate the possibility of having non-trivial geodesic parallelograms on the universal covers, which is the only case in which we would have more than one pair of critical points for the phase function $\phi(\alpha;t,s)$.

		\section{Gauss-Bonnet Theorem and Phase Function Bounds}
			 \begin{figure}
			\centering
			\includegraphics[width=.65\textwidth]{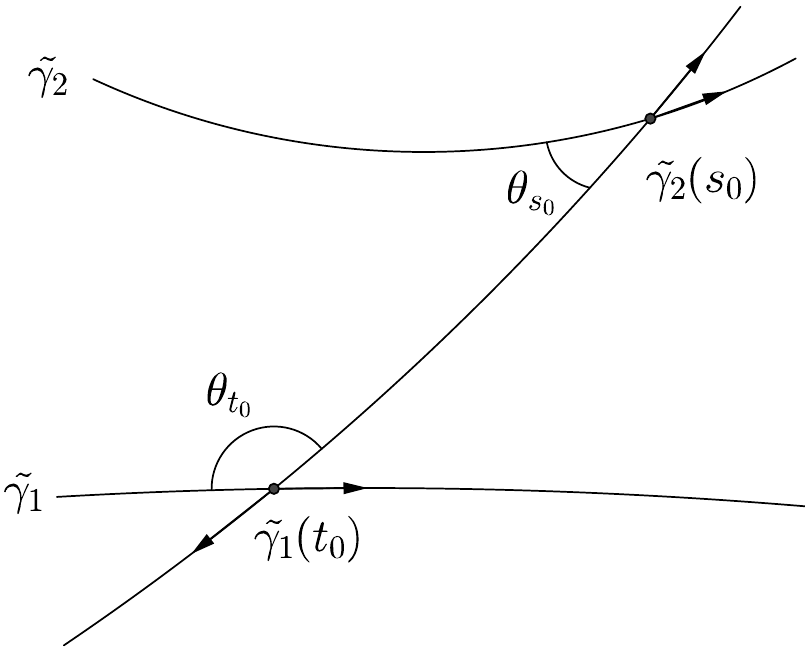}
			\caption{}
			\label{fig1}
		\end{figure}
		Let us now state the two geometric results that will play a key role in our proof. 	From now on, we shall take
		\begin{equation}\label{g.3}
		\e_0=\frac{1}{5k}
				\end{equation}
		where $k$ is the constant in \eqref{i.2}.

		\begin{proposition}\label{propg.1}  Let $\tilde \gamma_1(t)$ and $\tilde \gamma_2(s)$, $|s|, \, |t|\le 1/2$ be two  unit length geodesics in $({\mathbb R}^2,\tilde g)$ parameterized by arc-length, satisfying that $d_{\tilde g}(\tidle \gamma_1(t),\, \tilde \gamma_2(s))\ge 1$, $|t|, \, |s| \le 1/2$.  Let $\theta_{t}$ denote the angle made by the two unit vectors $\nabla_xd_{\tilde g}(\tilde \gamma_1(t),\tilde \gamma_2(s))$ and $\frac d{dt}{\tilde \gamma_1}(t)$. Similarly, let $\theta_{s}$ be the angle between the  two unit vectors $\nabla_yd_{\tilde g}(\tilde \gamma_1(t),\tilde \gamma_2(s))$ and $\frac d{ds}{\tilde \gamma_2}(s)$, see Figure \ref{fig1}. Fix a number $\epsilon$, such that $0\le\epsilon<c_0<1$.
	If we suppose further that there exists $(t_0,s_0)\in [-1/2,1/2]\times [-1/2,1/2]$ such that
			\begin{equation}\label{g.2}
			|\theta_{t_0}-\arccos\epsilon|+|\theta_{s_0}-\arccos(-\epsilon)| <\lambda^{-1/3},
			\end{equation}
			and if $\lambda$ is larger than a fixed constant, we have
			\begin{multline}\label{g.4}
			\max \bigl(|\theta_t-\arccos\epsilon|, \, |\theta_s-\arccos(-\epsilon)| \bigr) \ge \la^{-1/4}, 
			\\
			\text{if } \, \, t,s\in [-1/2,1/2] \, \, \text{and } \,  \max\bigl(|t-t_0|,\, |s-s_0|\bigr)\ge \la^{-\e_0}.
			\end{multline}
		\end{proposition}

		To prove Proposition~\ref{propg.1} we shall use a couple of special cases for the Gauss-Bonnet Theorem (see \cite{doCarmo})  concerning the sum of the interior angles
		$\alpha_j$ for geodesic quadrilaterals $Q$ and geodesic triangles ${\mathcal T}$ in $({\mathbb R}^2,\tg)$.  In the first case we define the ``defect'' of $Q$, $\text{Defect }Q$,
		to be $2\pi$ minus the sum of the four interior angles at the vertices, and in the case of ${\mathcal T}$, we define  $\text{Defect }{\mathcal T}$ to be $\pi$ minus the sum of its
		three interior angles, as shown in Figure~\ref{fig2}.
		
						 \begin{figure}
			\centering
			\includegraphics[width=0.95\textwidth]{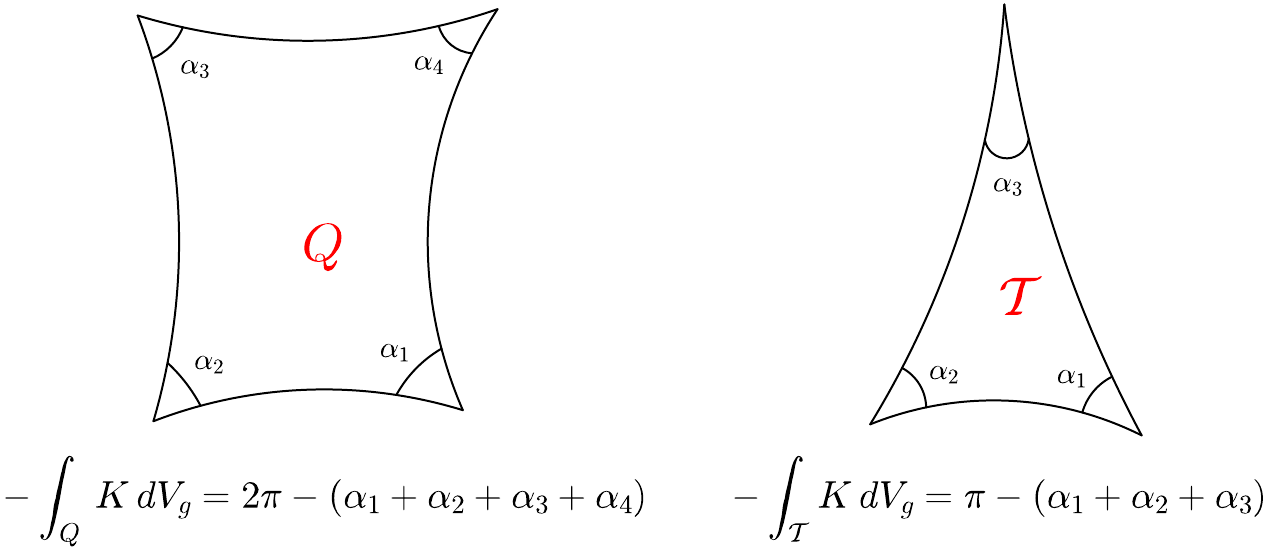}
			\caption{Gauss-Bonnet Theorem}
			\label{fig2}
		\end{figure}
		
		Then, the Gauss-Bonnet Theorem gives that
		\begin{gather*}
		\text{Defect }Q=-\int_Q K\, dV,
		\\
		\text{Defect }{\mathcal T}=-\int_{\mathcal T} K\, dV.
		\end{gather*}

		\begin{proof}[Proof of Proposition~\ref{propg.1}]
			Suppose that, given $(t_0,s_0)\in [-1/2,1/2]\times [-1/2,1/2]$, \eqref{g.2} is valid.  By symmetry it suffices to show
			that the conclusion in \eqref{g.4} is valid if we assume that $t\in [-1/2,1/2]\, \backslash \, (t_0-\la^{-\e_0},\, t_0+\la^{-\e_0})$ and
			$|s|\le 1/2$.  If $s\ne s_0$ there are two cases as shown in Figure~\ref{fig3} and Figure \ref{fig4}:  Either the geodesic segment connecting 
			$\tilde \gamma_1(t_0)$ and $\tilde \gamma_2(s_0)$ and the one connecting $\tilde \gamma_1(t)$ and $\tilde \gamma_2(s)$ intersect
			or not.  In the first case we obtain a geodesic quadrilateral $Q$ with vertices $\tilde \gamma_1(t_0), \, \tilde \gamma_1(t), \, \tilde\gamma_2(s_0)$
			and $\tilde\gamma_2(s)$, while in the other case we obtain two geodesic triangles using those four points  and the intersection
			point of the aforementioned geodesic segments.  To reach this conclusion we are using the fact that since we are assuming $K\le 0$,
			two geodesics in $({\mathbb R}^2,\tilde g)$ are disjoint or intersect at exactly one point by the Cartan-Hadamard theorem.
			
			In the first case, let $\alpha_{t_0}, \alpha_t,\alpha_{s_0}$ and $\alpha_s$ denote the interior angles of the geodesic quadrilateral $Q$
			at vertices $\tilde \gamma_1(t_0), \, \tilde \gamma_1(t), \, \tilde \gamma_2(s_0)$ and $\tilde \gamma_2(s)$, respectively.  Note that for the two interior angles $\alpha_{t_0}$ and $\alpha_{s_0}$, we have four possibilities depending on the orientations of $\tilde\gamma$ and $\alpha(\tilde{\gamma})$, i.e. the combinations of $\alpha_{t_0}=\theta_{t_0 }\text{ or }\pi-\theta_{t_0}$ and $\alpha_{s_0}=\theta_{s_0} \text{ or }\pi-\theta_{s_0}$. By possibly reversing the orientation of $\gamma$, we only need to consider two cases:  $\alpha_{t_0}=\theta_{t_0 }$, $\alpha_{s_0}=\pi-\theta_{s_0}$, or $\alpha_{t_0}=\theta_{t_0 }$, $\alpha_{s_0}=\theta_{s_0}$.
			
		1.	Now suppose that $\alpha_{s_0}=\pi-\theta_{s_0}$, and $\alpha_{t_0}=\theta_{t_0 }$ and $|t-t_0|>\la^{-\epsilon_0}$. Then by the Gauss-Bonnet theorem
			\begin{equation}\label{g.6}
			\text{Defect }Q=2\pi -\bigl(\alpha_{t_0}+\alpha_t+\alpha_{s_0}+\alpha_s\bigr)=-\int_QK\, dV.
			\end{equation}
			As in the first picture of Figure~\ref{fig3}, if we consider the geodesic ball, $B_r$, $r=\la^{-\e_0}/A$, which is tangent to $\tilde \gamma_1$ at $\tilde \gamma_1((t+t_0)/2)$ and
			on the same side of $\tilde \gamma_1$ as $Q$, it follows that, if $\la\gg1$, and $A$ is larger than a fixed constant depending on the metric and $c_0$, then we have
			$B_r\subset Q$.  We may make this assumption since otherwise we have that $\theta_t=\pi-\alpha_t$ is arbitrarily close to $\pi$, which trivially implies \eqref{g.4}.
			Now since $K\le0$, we have for large enough $\la$
			$$\text{Defect }Q=2\pi -\bigl(\alpha_{t_0}+\alpha_t+\alpha_{s_0}+\alpha_s\bigr)\ge -\int_{B_r}K\, dV \ge \delta' \la^{-k\e_0}=\delta' \la^{-1/5},$$
			where $\delta'=A^{-k}\delta$ for $\delta>0$ as in \eqref{i.2}.  Since we are assuming \eqref{g.2} we must have 	$|\alpha_{t_0}-\arccos\epsilon|, |\alpha_{s_0}-\arccos\epsilon|<\lambda^{-1/3},$ and
			notice that $\pi=\arccos\epsilon+\arccos(-\epsilon)$, we have
			\begin{multline}\label{pi-'}2\arccos(-\epsilon)-\alpha_t-\alpha_s\ge\alpha_{t_0}+\alpha_{s_0}-2\arccos\epsilon+\delta'\la^{-1/5}\\=\alpha_{t_0}-\arccos\epsilon+\alpha_{s_0}-\arccos\epsilon+\delta'\la^{-1/5}\ge-2\la^{-1/3}+\delta'\la^{-1/5}\ge \tfrac{\delta'}2 \la^{-1/5} \quad \text{if } \, \, \, \la \gg 1.\end{multline}
			Since $\theta_t=\pi-\alpha_t$, $\theta_s=\alpha_s$, we have $$2\arccos(-\epsilon)-\alpha_t-\alpha_s=2\arccos(-\epsilon)-\pi+\theta_t-\theta_s-\pi=\theta_t-\arccos\epsilon-(\theta_s-\arccos(-\epsilon)).$$
			Therefore, \eqref{pi-'} implies that
			\begin{equation}\label{g.7}
			\max \bigl(|\theta_t-\arccos\epsilon|, \, |\theta_s-\arccos(-\epsilon)| \bigr) \ge \la^{-1/4}, 
			\end{equation}
			if $\la$ is larger than a fixed constant which is independent of our two geodesic segments $\tilde \gamma_1$ and $\tilde \gamma_2$.
			
				 \begin{figure}[!ht]
				\centering
				\includegraphics[width=1\textwidth]{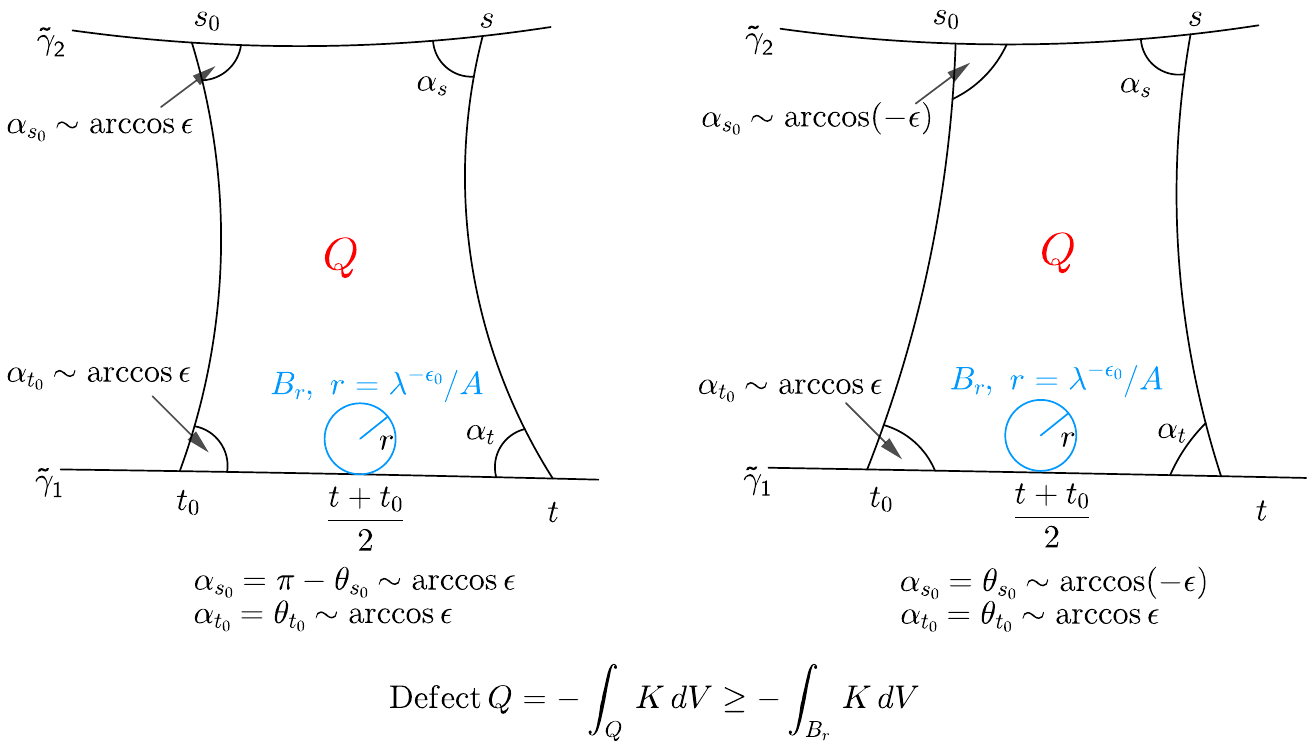}
				\caption{Non-intersecting Case}
				\label{fig3}
			\end{figure}
			
			2.	Now suppose that we are in the case that, $\alpha_{s_0}=\theta_{s_0}$, and $\alpha_{t_0}=\theta_{t_0 }$, again by Gauss-Bonnet Theorem,
			\begin{equation}\label{g.6'}
			\text{Defect }Q=2\pi -\bigl(\alpha_{t_0}+\alpha_t+\alpha_{s_0}+\alpha_s\bigr)=-\int_QK\, dV,
			\end{equation}
			see the second picture in Figure \ref{fig3}. Then by the same argument as above, the geodesic ball, $B_r$, $r=\la^{-\e_0}/A$, will be contained in $Q$ if $\la$ and $A$ are sufficiently large. Therefore, we have 
			$$\text{Defect }Q=2\pi -\bigl(\alpha_{t_0}+\alpha_t+\alpha_{s_0}+\alpha_s\bigr)\ge -\int_{B_r}K\, dV \ge \delta' \la^{-k\e_0}=\delta' \la^{-1/5},$$
			where $\delta'=A^{-k}\delta$ for $\delta>0$ as in \eqref{i.2}.  
			
			 Since we are assuming \eqref{g.2} we must have 	$|\alpha_{t_0}-\arccos\epsilon|, |\alpha_{s_0}-\arccos\epsilon|<\lambda^{-1/3},$ and
			notice that $\pi=\arccos\epsilon+\arccos(-\epsilon)$, we have
			\begin{multline}\label{pi-''}2\arccos(-\epsilon)-\alpha_t-\alpha_s\ge\alpha_{t_0}+\alpha_{s_0}-2\arccos\epsilon+\delta'\la^{-1/5}\\=\alpha_{t_0}-\arccos\epsilon+\alpha_{s_0}-\arccos\epsilon+\delta'\la^{-1/5}\ge-2\la^{-1/3}+\delta'\la^{-1/5}\ge \tfrac{\delta'}2 \la^{-1/5}, \quad \text{if } \, \la \gg 1.\end{multline}
			Notice that $\theta_t=\pi-\alpha_t$, $\theta_s=\alpha_s$, we have $$2\arccos(-\epsilon)-\alpha_t-\alpha_s=2\arccos(-\epsilon)-\pi+\theta_t-\theta_s-\pi=\theta_t-\arccos\epsilon-(\theta_s-\arccos(-\epsilon)).$$
			Therefore, \eqref{pi-''} implies that
			\begin{equation}\label{g.7'}
			\max \bigl(|\theta_t-\arccos\epsilon|, \, |\theta_s-\arccos(-\epsilon)| \bigr) \ge \la^{-1/4}, 
			\end{equation}
			if $\la$ is larger than a fixed constant depending only on the metric.

		Now we are left with the case when the geodesics connecting
			$\tilde \gamma_1(t_0)$ and $\tilde \gamma_2(s_0)$ and the one connecting $\tilde \gamma_1(t)$ and $\tilde \gamma_2(s)$ intersect at a point $P$.  Recall that we are assuming $|t-t_0|>\la^{-\epsilon_0}$. Let us consider the geodesic triangle ${\cal T}$ with vertice $\tilde \gamma_1(t_0), \,\tilde \gamma_1(t)$ and $P$.  If
			$\alpha_{t_0}, \, \alpha_t$ and $\alpha_P$ are the corresponding interior angles for ${\mathcal T}$, as before, if $\la,\, A$ are large enough the geodesic
			ball $B_r$, $r=\la^{-\e_0}/A$, which is tangent to $\tilde \gamma_1$ at $\tilde \gamma_1((t+t_0)/2)$ and on the same side as ${\mathcal T}$ must be contained
			in ${\mathcal T}$. Therefore, if $\delta'=A^{-k}\delta$ for $\delta>0$ as in \eqref{i.2}, we have 
			$$\pi-(\alpha_{t_0}+\alpha_t+\alpha_P)=-\int_{\mathcal T} K\, dV\ge -\int_{B_r} K\, dV \ge \delta' \la^{-k\e_0}=\delta' \la^{-1/5}.$$
			See Figure \ref{fig4}.
			 \begin{figure}[!ht]
				\centering
				\includegraphics[width=0.7\textwidth]{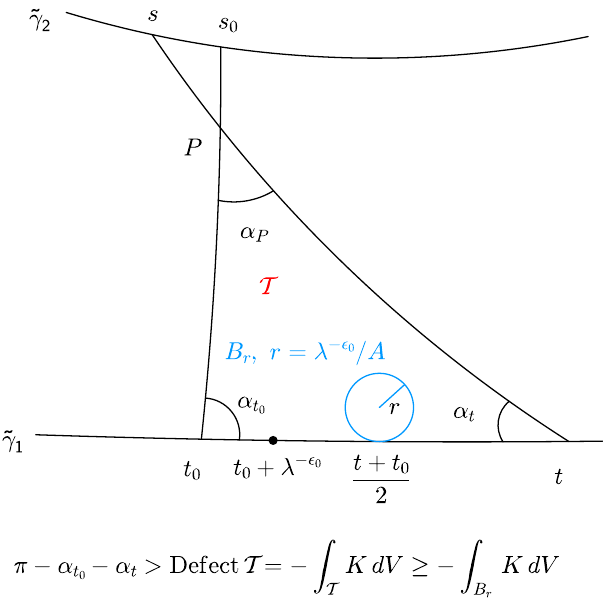}
				\caption{Intersecting case}
				\label{fig4}
			\end{figure}
			Thus, on one hand, if $\alpha_{t_0}=\theta_{t_0}$, $\alpha_{t}=\pi-\theta_{t}$ 
			\begin{multline}\theta_t-\arccos\epsilon=\arccos\epsilon-\alpha_t\ge\arccos\epsilon-(\alpha_t+\alpha_P)\ge\alpha_{t_0}-\arccos(-\epsilon)+\delta'\la^{-1/5}\\ = \theta_{t_0}-\arccos(-\epsilon)+\delta'\la^{-1/5}\ge\tfrac{\delta'}2\la^{-1/5}, \quad \text{if } \, \, \la\gg 1.\end{multline}
			On the other hand, if $\alpha_{t_0}=\pi-\theta_{t_0}$, $\alpha_{t}=\theta_{t},$
			 \begin{multline}\arccos\epsilon-\theta_t=\arccos\epsilon-(\alpha_t+\alpha_P)\ge\alpha_{t_0}-\arccos(-\epsilon)+\delta'\la^{-1/5}\\=\arccos\epsilon-\theta_{t_0}+ \delta'\la^{-1/5}\ge\tfrac{\delta'}2\la^{-1/5},\quad \text{if } \, \, \la\gg 1,
			 \end{multline}
			and therefore, in both cases, $|\theta_t-\arccos\epsilon|\ge \la^{-1/4}$, and thus \eqref{g.4} is valid as well, 
			completing the proof.\end{proof}

		A direct consequence of Proposition \ref{propg.1} is the following estimates for our phase function $\phi$ in the non-critical case:
				\begin{proposition}\label{prop5.3}  Let $\phi(\alpha;t,s)$  be as above with $\alpha \ne Id$.   Then if for $t_0,s_0\in [-1/2,1/2]$ and  $\la\gg 1$
			\begin{equation}\label{5.26}
			\bigl|\nabla_{t,s}\phi(\alpha;t_0,s_0)\bigr| \le \frac12 \la^{-1/3},
			\end{equation}
			it follows that if $\e_0$ is as in \eqref{g.3}, we have
			\begin{equation}\label{5.27}
			\bigl|\nabla_{t,s}\phi(\alpha;t,s)\bigr|\ge \frac12 \la^{-1/4}, \quad \text{if } \, \, \, 
			\max\bigl(|t-t_0|, \, |s-s_0|\bigr)\ge \la^{-\e_0} \, \, \text{and } \, |t|,|s|\le 1/2.
			\end{equation}
		\end{proposition}
		\begin{proof}
		Note that by \eqref{5.18}, if we take $\tilde\gamma$ and $\alpha(\tilde\gamma)$ to be $\tilde\gamma_1$ and $\tilde\gamma_2$ respectively, $\partial_t\phi(\alpha;t,s)=\cos \theta_t-\epsilon$, and $\partial_s \phi(\alpha;t,s)=\cos \theta_s+\epsilon$, where $\theta_t$ and
		$\theta_s$ are as in Proposition~\ref{propg.1}.  From this one immediately sees that Proposition~\ref{prop5.3} is a consequence of
		Proposition~\ref{propg.1} .  \end{proof}
		
	Proposition \ref{prop5.3} provides  a nice lower bound for the gradient of our phase function in the case that $s,\,t$ are quantitatively away from a pair of critical points. Thus it remains to obtain a suitable lower bound for the second derivatives when we are at the near-critical situation. Such a lower bound has been proven in \cite{Gauss} for a similar situation.
		
		\begin{proposition}[{\cite[Proposition 5.2]{Gauss}}]\label{prop5.2}  Let $\phi_\pm(\alpha;t,s)$ be as in \eqref{5.18} with $\alpha\ne Id$.  Then for each $j=1,2,3,\dots$ there is a constant $C_j$ so that
			\begin{equation}\label{5.19}
			|\partial_t^j\phi(\alpha;t,s)|+|\partial_s^j\phi(\alpha;t,s)|\le \exp(C_jT),
		\quad\text{if } \, \,d_{\tilde g}(\tilde \gamma(t),\alpha(\tilde \gamma(s)))\le T.
			\end{equation}
			Moreover, we have the uniform bounds
			\begin{equation}\label{5.20}
			|\partial_t\partial_s\phi(\alpha;t,s)|\le C.
			\end{equation}
			Additionally,
			\begin{equation}\label{5.21}
			\partial^2_t\phi(\alpha;t_0,s_0) \ge \exp(-CT), 
			\end{equation}
			for some $C$ if
			\begin{equation}\label{5.22}
			d_{\tilde g}(\tilde \gamma(t_0),\alpha(\tilde \gamma(s_0)))\le T	\quad \text{and } \, \, |\partial_t\phi(\alpha;t_0,s_0)|<c_1,
			\end{equation}
			where $c_1>0$ is a constant that only depends on $c_0$.
		\end{proposition}
		The proof of proposition \ref{prop5.2} is identical to that of  \cite[Proposition 5.2]{Gauss}, due to the fact that the two phase functions are only differed by a term which is  linear in both $s$ and $t$, so that  they have identical second order derivatives. The only notable  difference is that here we shall need $|\partial_t\phi(\alpha;t_0,s_0)|=|\cos \theta_t-\epsilon|$ to be sufficiently small, so that we can get a positive lower bound for $\sin\theta_t$, which is essential in the proof of \cite[Proposition 5.2]{Gauss}. We shall omit the proof here.

		\section{End of proof of Theorem \ref{main} and proof of Corollary \ref{co}}
		The last ingredient of our proof is the following stationary phase lemma  developed in \cite{Gauss}, which is a more flexible version of the standard stationary phase arguments. Note that as in \cite{Gauss} we are not seeking for the optimal power gain in the following lemma, since any power gain would suffice.

		\begin{lemma}[{\cite[Proposition 4.3]{Gauss}}]\label{stationary}  Suppose that $\phi\in C^\infty(\R)$ is real valued and that $a\in C^\infty_0({\mathcal I})$, where ${\mathcal I}\subset (-1/2,1/2)$ is an open interval.  Suppose that
			for some $0<\sigma<1/2$
			\begin{equation}\label{s.12}
			\la^{-\sigma/2}\le |\phi''(t)|, \quad t\in {\mathcal I}.
			\end{equation}
			Suppose further that for $0\le j\le N=\left \lceil{4\sigma^{-1}}\right \rceil$
			\begin{equation}\label{s.13}
			|\partial^j_ta(t)|\le C_j\la^{j/2}
			\end{equation}
			and that
			\begin{equation}\label{s.14}
			|\partial^j_t \phi'|\le \la^{\sigma/2}, \quad t\in{\mathcal I}.
			\end{equation}
			Then
			\begin{equation}\label{s.15}
			\Bigl|\, \int e^{i\la \phi(t)} \, a(t)\, dt \, \Bigr|\le C\la^{-1/2+2\sigma},
			\end{equation}
			where $C$ depends only on $\sigma$ and the above constants $C_j$, $j\le \left \lceil{4\sigma^{-1}}\right \rceil$.
		\end{lemma}
	Finally, we are ready to finish the proof of Proposition \ref{prop2.1}, hence completing the proof of Theorem \ref{main}.
		\begin{proof}[Proof of Proposition \ref{prop2.1}]Now we are in a position to finish the proof of Proposition~\ref{prop2.1} and hence that of Theorem~\ref{main}.  We need to verify that we can fix
		$c=c_M>0$ so that for $\la\gg1$, \eqref{2.21} is valid for some $\sigma_M>0$.  We shall take 
		$$\sigma_M=\e_0/2$$
		where $\e_0=\frac{1}{5k}>0$ is as in \eqref{g.3} and \eqref{5.27}.

	Thus it suffices to show that for $\alpha\neq Id$, we have
		\begin{equation}\label{6.2}
		\la^{1/2}\Bigl| \, \iint b(t,s)a(T,\la; \tilde \gamma(t), \alpha(\tilde \gamma(s))) \, e^{ i \la\phi(\alpha;t,s)}\, dt ds\, \Bigr|
		\le C_{M,b}\la^{-\e_0/2},
		\end{equation}
		under the above assumptions with $\phi(\alpha;t,s)=\epsilon(s-t)+d_{\tilde g}(\tilde \gamma(t),\alpha(\tilde \gamma(s)))$ as in 
		Propositions \ref{prop5.2} and \ref{prop5.3}. 
		
		Therefore, by \eqref{5.4}, \eqref{5.5} and \eqref{5.19} if $T=c\log\la$ with $c=c_M>0$ sufficiently small and if $(\tilde\gamma(t),\alpha(\tilde \gamma(s))$ is in the support of $K_{T,\la}$, we have
		\begin{multline}\label{6.3}
		|\partial^j_t\phi(\alpha;s,t)|+|\partial^j_s\phi(\alpha;s,t)|+|\partial_t^ja(T,\la; \tilde \gamma(t),\alpha(\tilde \gamma(s)))|
		+|\partial_s^ja(T,\la; \tilde \gamma(t),\alpha(\tilde \gamma(s)))|
		\\
		\le \la^{\e_0/8}, \quad \text{with } \, \, \, 1 \le j\le \left \lceil{8\e_0^{-1}}\right \rceil.
		\end{multline}
		We use $\e_0/8$ here since we intend to later apply Lemma~\ref{stationary} with $\sigma=\e_0/4$.
		
		To apply Proposition~\ref{prop5.3}, \ref{prop5.2} and the stationary phase lemma, we shall consider two cases:
		\begin{equation}\label{6.4}
		|\nabla_{t,s}\phi(\alpha;s,t)|\ge \frac12 \la^{-1/3}, \quad |t|, \, |s| <1/2,
		\end{equation}
		and the complementary case where
		\begin{equation}\label{6.5}
		|\nabla_{t,s}\phi(\alpha;t_0,s_0)|\le \frac12 \la^{-1/3} \quad \text{for some } \, \, (t_0,s_0)\in (-1/2)\times (-1/2).
		\end{equation}
		
		To show that \eqref{6.2} is valid under the assumption \eqref{6.4} we shall use a partition of unity argument to exploit \eqref{6.3}.  Specifically,
		choose $\rho\in C^\infty_0((-1,1))$ satisfying 
		$$\sum_{j=-\infty}^\infty \rho(t-j)\equiv 1, \quad t\in \R.$$
		Then for $m=(m_1,m_2)\in {\mathbb Z}^2$ set
		$$\rho_m(t,s)=\rho(\la^{1/2}t-m_1) \, \rho(\la^{1/2}s-m_2).$$
		It follows that $\sum_{m\in {\mathbb Z}^2}\rho_m(t,s)\equiv 1$ and that $|\partial_t^j\rho_m| +|\partial_s^j\rho_m|\le C_j\la^{j/2}$.  Also, $\rho_m$
		is supported in a $O(\la^{-1/2})$ size neighborhood about $(t_m,s_m)=(\la^{-1/2}m_1,\la^{-1/2}m_2)$.  Assuming that this neighborhood intersects
		$(-1/2,1/2)\times (-1/2,1/2)$ and that $(\overline{t_m},\overline{s_m})$ is in the intersection, by \eqref{6.4} we must have that
		$$|\partial_t\phi(\alpha;\overline{t_m},\overline{s_m})|\ge \frac14\la^{-1/3}, 
		\, \, \, \text{or } \, \, |\partial_s\phi(\alpha;\overline{t_m},\overline{s_m})|\ge \frac14\la^{-1/3}.$$
		Let us assume the former since the argument for the latter is similar.  By \eqref{5.20} and \eqref{6.3} we must have that
		$$|\partial_t\phi|\ge \frac14\la^{-1/3}-O(\la^{-1/2}\la^{\e_0/4})-O(\la^{-1/2})\ge \frac18\la^{-1/3} \quad
		\text{on  supp }\alpha_m$$
		for large $\la$ since $\e_0<1/10$.  Therefore, by a simple integration by parts argument, we have 
			$$\la^{1/2}\Bigl| \, \int b(t,s)\rho_m(t,s)a(T,\la; \tilde \gamma(t), \alpha(\tilde \gamma(s))) \, e^{ i\la \phi(\alpha;t,s)} \, dt \, \Bigr|
		\le C_{b,\e_0}\la^{-3/2}.$$
		This is because every time we do an integration by parts in the $t$ variable, we gain a factor of $\la^{-\frac16}$.
		
		Since $\rho_m(t,s)=0$ if $|s-s_m|\ge C\la^{-1/2}$, this in turn gives the bounds
		$$\la^{1/2}\Bigl| \, \iint b(t,s)\rho_m(t,s)a(T,\la; \tilde \gamma(t), \alpha(\tilde \gamma(s))) \, e^{ i\la \phi(\alpha;t,s)} \, dtds \, \Bigr|
		\le C_{b,\e_0}\la^{-2}.$$
		Since there are $O(\la)$ such terms which are nonzero, we conclude that when \eqref{6.4} holds we obtain a stronger version of 
		\eqref{6.2} where $\la^{-\e_0/2}$ could be replaced by $\la^{-1}$.
		
		Now let us assume \eqref{6.5}. Then it suffices to show that \eqref{6.2} is valid in this case too.  To achieve this, we shall use
		Proposition~\ref{prop5.3}, which makes use of our curvature assumption \eqref{i.2}. Let $\beta\in C^\infty_0(\R)$ be as in \S 2, i.e.,
		$$\beta(t)=1 \, \, \text{for } \, \, |t|\le3 \, \, \, \text{and } \, \, \, \beta(t)=0 \, \, \, \text{for } \, \, |t|\ge4.$$
			Since by \eqref{5.27} we have,
		$$|\nabla_{t,s}\phi(\alpha;t,s)|\ge \frac12\la^{-1/4}, \quad 
		\text{if } \, \, \bigl(1-\beta(\la^{\e_0}|t-t_0|)\beta(\la^{\e_0}|s-s_0|)\bigr)\ne 0,$$
		then integration by parts yields
		\begin{multline*}
		\la^{1/2}\Bigl|\, \iint \bigl(1-\beta(\la^{\e_0}|t-t_0|)\beta(\la^{\e_0}|s-s_0|)\bigr) \, a(T,\la; \tilde \gamma(t), \alpha(\tilde \gamma(s)))
		\, e^{ i\la\phi(\alpha;t,s)} \, dt ds \, \Bigr|
		\\
		\le C_{b,\e_0}\la^{-1},
		\end{multline*}

		Thus, the proof of our main theorem would be complete if we could show that
		\begin{multline}\label{6.6}
		\la^{1/2}\Bigl|  \iint  \beta(\la^{\e_0}|t-t_0|)\beta(\la^{\e_0}|s-s_0|) \, b(t,s) \, a(T,\la; \tilde \gamma(t), \alpha(\tilde \gamma(s))) \,
		e^{ i \la\phi(\alpha:t,s)} \, dt ds \Bigr| 
		\\
		\le C_{M,b}\la^{-\e_0/2}.
		\end{multline}
		To do this, we note that, by \eqref{5.21}, if $T=c\log\la$ with $c=c_M>0$ sufficiently small, we have
		$$|\partial_t^2\phi(\alpha;t,s)|\ge \la^{-\e_0/8} \quad \text{if } \, \, 
		\beta(\la^{\e_0}|t-t_0|)\beta(\la^{\e_0}|s-s_0|)\ne 0,$$
		since our assumption \eqref{6.5} along with \eqref{6.3} and \eqref{5.20} ensures that \eqref{5.22} is valid for such $(t,s)$.
		Therefore, applying Lemma~\ref{stationary} with $\sigma=\e_0/4$, we have that for each $s\in[-1/2,1/2]$,
		\begin{multline}\label{last}
		\la^{1/2}\Bigl|  \int  \beta(\la^{\e_0}|t-t_0|)\beta(\la^{\e_0}|s-s_0|) \, b(t,s) \, a(T,\la; \tilde \gamma(t), \alpha(\tilde \gamma(s))) 
		e^{ i \la \phi(\alpha:t,s)} \, dt   \Bigr| 
		\\
		\le C_{M,b}\la^{\e_0/2}.
		\end{multline}
		Since the $s$-support of $ \beta(\la^{\e_0}|s-s_0|)$ has size about $\la^{-\e_0}$, \eqref{last} implies \eqref{6.6}, which completes
		the proof.  
		\end{proof}
		Finally, we give a simple proof to Corollary \ref{co}.
\begin{proof}[Proof of Corollary \ref{co}.]
	For the sake of simplicity, we may assume $|\gamma|=2\pi$, and the $L^2(\gamma)$ norm of $e_{\la_j}$ is bounded by 1. Let $g\in L^2(\gamma),$ it then suffices to show that given any $\eps>0$,  for large enough $j$, we have
	\[\left|\int_\gamma e_{\la_j}(\gamma(s))\,g(s)\,ds\right|\le \eps,\]
	Now we write $g$ in terms of its Fourier series,
	\[g(s)=\sum_ka_ke^{iks}.\]
	$g\in L^2(\gamma)$ implies that there exists some $N>0$, such that
	\[\sum_{|k|>N}|a_k|^2\le\frac14\eps^2.\]
	If we let
		\[b_{j,k}=\left|\int_\gamma e_{\la_j}\,e^{iks}\,ds\right|,\]
	then by Theorem \ref{main}, 
	\[b_{j,k}\le C\frac{1}{(\log\la_j)^\frac12},\]
	here $C$ will be uniform provided that $|k|\le\frac12\la_j.$
	
	Thus	
	\[\left|\int_\gamma e_{\la_j}\,g\,ds\right|\le\sum_{|k|\le N}|a_k||b_{k,j}|+\Big[\sum_{|k|>N}|a_k|^2\Big]^\frac12\|e_{\la_j}\|_{L^2(\gamma)},\]
	while the first sum on the right hand side is bounded by
	
	\[C{(\log\la_j)^{-\frac12}}\sum_{|k|\le N}|a_k|\le C(\log\la_j)^{-\frac12}N^\frac12[\sum_{|k|\le N}|a_k|^2]^\frac12.\]
	
	Thus, if $\log{\la_j}\ge[{4\eps^{-2}NC^2\|g\|^2_{L^2(\gamma)}}]+1,$ we have
	
	\[\left|\int_\gamma e_{\la_j}\,g\,ds\right|\le \eps.\]
	This choice of $j$ is valid, due to the fact that $N/\la_j$ would be comparable to $\eps^2\la_j^{-1}\log\la_j\ll 1/2$, which makes the constant $C$ uniform.
\end{proof}
		\bibliography{EF}{}
		\bibliographystyle{alpha}
	\end{document}